\documentclass[preprint,12pt,authoryear]{elsarticle}

\usepackage[T1]{fontenc}
\usepackage{amsmath,amssymb,amsthm,mathtools}
\usepackage{bm}
\usepackage{mathrsfs}
\usepackage{xfrac}
\usepackage{extarrows}
\usepackage{graphicx}
\usepackage{booktabs}
\usepackage{enumerate}
\usepackage{enumitem}
\usepackage{xcolor}

\usepackage{natbib}
\setcitestyle{numbers,square}
\usepackage{hypernat}
\usepackage[hidelinks]{hyperref}
\hypersetup{
  colorlinks=true,
  citecolor=cyan,
  linkcolor=cyan,
  urlcolor=cyan
}

\everymath{\displaystyle}
\numberwithin{equation}{section}

\theoremstyle{plain}
\newtheorem{theorem}{Theorem}[section]
\newtheorem{proposition}[theorem]{Proposition}
\newtheorem{lemma}[theorem]{Lemma}
\newtheorem{corollary}[theorem]{Corollary}

\theoremstyle{definition}
\newtheorem{definition}[theorem]{Definition}

\newtheoremstyle{myremark}
  {6pt}
  {6pt}
  {\normalfont}
  {}
  {\bfseries}
  {.}
  {0.5em}
  {}

\theoremstyle{myremark}
\newtheorem{remark}[theorem]{Remark}

\journal{\textcolor{red}{}}

\begin{document}

	\begin{frontmatter}
	
		\title{Hilbertian Hardy--Sobolev Spaces on Tube Domains over Convex Cones}
		
		\author[author1]{Haichou Li}
        \author[author2]{Tao Qian}

		\address[author1]{College of Mathematics and Informatics, South China Agricultural University, Guangzhou, China. Email: hcl2016@scau.edu.cn}
		\address[author2]{Corresponding author. Faculty of Innovation Engineering, Macau University of Science and Technology, Macau, China. Email: tqian@must.edu.mo} 


\begin{abstract}

We introduce Hilbertian Hardy--Sobolev spaces on tube domains over convex cones
and develop their structural theory from a Fourier-analytic point of view.
We first establish a Paley--Wiener type representation, which identifies these
spaces with weighted $L^2$ spaces on the dual cone and reveals their intrinsic
Fourier structure.
This representation leads naturally to a Hardy--Sobolev decomposition theorem
for boundary Sobolev spaces on $\mathbb{R}^d$.
Building on these structural results, we derive explicit reproducing kernels
and characterize Carleson measures for the Hilbertian Hardy--Sobolev spaces.
As a preliminary operator-theoretic application, we also derive
basic consequences for multipliers and weighted composition operators on
these spaces.
\end{abstract}
   

	\end{frontmatter}

\tableofcontents
\section{Introduction}

Hardy--Sobolev spaces form a natural class of analytic function spaces
that combine the boundary control of Hardy spaces with the regularity
features measured by Sobolev norms. Roughly speaking, they consist of
holomorphic functions whose boundary behaviour is controlled in the
Hardy sense, while their smoothness is measured through Sobolev-type
conditions. Such spaces arise naturally at the intersection of complex
analysis, harmonic analysis, and operator theory, and may be viewed as
regularity refinements of classical Hardy spaces.

From a general function-space perspective, the Sobolev component of
Hardy--Sobolev spaces is closely related to derivatives, Bessel
potentials, and Fourier weights. In particular, Sobolev regularity may
be characterized through polynomial weights in the Fourier variable,
reflecting the basic principle that the Fourier transform converts
differentiation into multiplication. Standard references for this
Fourier characterization of Sobolev spaces include
Stein~\cite{Stein1970}, Adams and Fournier~\cite{AdamsFournier2003},
and Triebel~\cite{Triebel1983,Triebel1978}.

Hardy--Sobolev spaces admit several natural formulations in the
literature. A common definition is derivative-based, requiring that a
holomorphic function $F$ satisfy
\[
\partial^\alpha F\in H^2(D),
\quad |\alpha|\le n,
\]
where $D$ denotes the underlying domain. Such derivative-type
formulations have been studied extensively in various settings; see, for
instance,
\cite{CascanteOrtega1995,OrtegaFabrega2006,HeCaoZhu2015,LiuLiKouQu2025}.
Other approaches describe Hardy--Sobolev spaces through boundary
Sobolev regularity or through Fourier-analytic conditions involving
weighted Fourier transforms.

On classical bounded domains such as the unit disk and the unit ball in
$\mathbb C^n$, Hardy--Sobolev spaces have by now been studied rather
systematically. Major themes of the theory include Carleson-type
embeddings, multiplier algebras, spectral questions, and composition
operators. Atomic decompositions were established by Cho and
Kim~\cite{ChoKim2006}, while operator-theoretic questions were studied
in a number of works including
\cite{CaoHe2014,CaoHeZhu2018,HeCaoZhu2015}.
These results show that Hardy--Sobolev spaces support a substantial
function theory and operator theory in the bounded-domain setting.

Compared with bounded domains, the theory on unbounded domains remains
much less developed. Even in one complex variable, systematic studies of
Hardy--Sobolev spaces on the upper half-plane have emerged only
relatively recently. In the Hilbertian setting,
Gal\'e et al.~\cite{GaleMatacheMiana2020}
introduced Hardy--Sobolev spaces on the half-plane via an extended
Paley--Wiener theorem and derived explicit reproducing kernels.
Further developments on Hardy--Sobolev spaces and related operator
theory can also be found in the one-dimensional setting. More precisely,
\cite{LiuLiKouQu2025} studied the upper half-plane in the Hilbert case
$p=2$ by using a weighted derivative-type definition, whereas
\cite{HaoxianFirstPaper} treated the full range $1\le p\le\infty$ from a
derivative-based viewpoint. These works provide useful one-dimensional
background for the present study, while the high-dimensional tube-domain
setting considered here requires a substantially different framework
adapted to the cone geometry and the Fourier--Laplace representation.

In several complex variables, Hardy--Sobolev spaces have also been
studied on domains whose geometry plays an essential role.
Rotkevich obtained constructive descriptions of Hardy--Sobolev spaces
on strongly convex domains and strictly pseudoconvex domains
\cite{Rotkevich2018,Rotkevich2022}.
Another important higher-dimensional unbounded setting is the Siegel
upper half-space, where Arcozzi et al.\ proved
Paley--Wiener type theorems for weighted Dirichlet-type spaces
\cite{ArcozziMonguzziPelosoSalvatori2019}, and Chalmoukis and Lamberti
established Carleson measure criteria for Hardy--Sobolev spaces in this
framework~\cite{ChalmoukisLamberti2025}.
These developments indicate that Hardy--Sobolev phenomena on unbounded
domains are closely tied to the geometry of the ambient domain.

The present paper is concerned with tube domains over convex cones.
Let $\Omega\subset\mathbb R^d$ be an open convex cone in the
$d$-dimensional Euclidean space $\mathbb R^d$. Consider the
associated tube domain
\[
T_\Omega=\mathbb R^d+i\Omega .
\]
Such domains arise naturally in Fourier analysis and several complex
variables and provide a natural setting for analytic function spaces
adapted to cone geometry. Hardy spaces on tube domains admit a
well-developed Fourier--Laplace description in terms of spectral
support in the dual cone $\Omega^\ast$; see, for instance,
Garrig\'os~\cite{Garrigos2001},
Li, Deng and Qian~\cite{LiDengQian2018,DengLiQian2019}.
This spectral characterization suggests a natural way to combine Hardy
structure with Sobolev regularity.

Related developments in Hardy--Sobolev analysis also appear in the
literature on decomposition, derivative-related questions and signal
analysis; see, for example,
\cite{DangQianYou2011,DangQianYang2012}.
More recent work closer to the present function-space and
operator-theoretic viewpoint includes
\cite{LiuLiKouQu2025,HaoxianFirstPaper}.
These works indicate that Hardy--Sobolev spaces admit a rich structure
from both function-theoretic and applied perspectives, although the
corresponding theory on high-dimensional tube domains remains far from
complete.

At the Hardy-space level, related Fourier spectrum characterizations
and decomposition results have been established in several settings.
For Hardy spaces on tubes over cones, Fourier spectrum
characterizations were obtained for $1\le p\le\infty$ and later
extended to the range $0<p<1$; see
\cite{LiDengQian2018,DengLiQian2019}.
On the other hand, decomposition results were proved for
$L^p(\partial\mathbb D)$, $0<p\le 1$, on the unit circle and for
$L^p(\mathbb R^n)$, $0<p<1$, in higher dimensions via Hardy spaces on
tubes; see \cite{LiDengQian2016,DengLiQianCVEE2019}.
Although these works concern Hardy spaces rather than Hardy--Sobolev
spaces, they provide a natural analytic background for the
Fourier--Laplace viewpoint adopted here and for the decomposition
result established later in the paper.

In the present tube-domain setting, this Hardy structure is naturally
reflected in the requirement that the Fourier spectrum be supported in
the dual cone $\Omega^\ast$. On the other hand, classical Sobolev regularity on
$\mathbb R^d$ is described in Fourier terms by polynomial weights in
the frequency variable. It is therefore natural, in the present
geometric framework, to define a Hardy--Sobolev space by combining
these two features: one keeps the spectral support condition
characteristic of Hardy spaces on tube domains and imposes, in
addition, a Sobolev-type weighted $L^2$ condition on the corresponding
Fourier transform. In this sense, the definition adopted in the present
paper may be viewed as a Fourier--Laplace realization of
Hardy--Sobolev regularity adapted to the geometry of convex-cone tube
domains.

Motivated by this observation, in the present paper we introduce a
Hilbertian Hardy--Sobolev space on $T_\Omega$ defined through a
Fourier--Laplace representation. More precisely, functions in the
Hilbertian Hardy--Sobolev space considered here admit representations
of the form
\[
F(z)=\int_{\Omega^\ast} e^{i\langle z,\xi\rangle} f(\xi)\,d\xi ,
\]
where the spectral function $f$ satisfies the weighted condition
\[
\int_{\Omega^\ast}|f(\xi)|^2(1+\rho(\xi)^2)^n\,d\xi<\infty .
\]
In this formulation, the Hardy aspect is reflected in the spectral
support condition in $\Omega^\ast$, while the Sobolev regularity is
measured through polynomial weights in the Fourier variable. The
precise definition of the Hilbertian Hardy--Sobolev spaces studied in
this paper will be given in Section~3.

The purpose of the present paper is to develop a structural theory of
Hilbertian Hardy--Sobolev spaces on tube domains over convex cones.
Our analysis combines tools from Fourier analysis, several complex
variables, convex geometry, and operator theory.
We first introduce the Hilbertian Hardy--Sobolev spaces and establish
their basic Hilbert space structure. We then prove a Paley--Wiener type
characterization describing the spaces in terms of weighted Fourier
transforms supported in the dual cone.

Next we establish a Hardy--Sobolev decomposition theorem showing that
boundary Sobolev functions with spectrum supported in
$\Omega^\ast\cup(-\Omega^\ast)$ split into Hardy--Sobolev components
associated with the opposite tube domains $T_\Omega$ and $T_{-\Omega}$.
We also derive explicit reproducing kernel formulas for the spaces and
use them to characterize Carleson measures through kernel testing
conditions. Finally, we derive several preliminary operator-theoretic
consequences of the structural results developed in the paper,
including basic facts on multipliers and weighted composition operators
on the Hilbertian Hardy--Sobolev spaces.

These results obtained in the perent paper are not isolated, 
but form a coherent progression: the Paley--Wiener characterization provides the structural foundation for
the Hardy--Sobolev decomposition, the reproducing-kernel theory and the
Carleson-measure characterization, while the latter in turn lead
naturally to the preliminary operator-theoretic consequences discussed
at the end of the paper.

\subsection*{Organization of the paper}

The paper is organized as follows.
Section~2 recalls the necessary background on convex cones,
tube domains,  Hardy spaces on tubes and Sobolev spaces on $\mathbb{R}^d$.
In Section~3 we introduce the Hilbertian Hardy--Sobolev spaces that
form the main object of study.
Section~4 establishes a Paley--Wiener type characterization.
Section~5 is devoted to a Hardy--Sobolev decomposition theorem.
In Section~6 we derive reproducing kernel formulas.
Section~7 characterizes Carleson measures.
Finally, Section~8 derives several preliminary operator-theoretic
consequences concerning multipliers and weighted composition operators
on the Hardy--Sobolev spaces.

\section{Preliminaries}

In this section we recall several geometric and Fourier--analytic notions
related to tube domains over convex cones that will be used throughout
the paper. Standard background on convex cones, tube domains, and
Fourier analysis in this setting may be found, for instance, in
\cite{FarautKoranyi1994,SteinWeiss1971,LiDengQian2018}.
Some additional notions concerning Sobolev spaces are also recalled
for later use.

\subsection{Convex cones and tube domains}

Let $\Omega\subset\mathbb{R}^d$ be an open convex cone. Throughout the
paper we assume that $\Omega$ is nonempty and does not contain any
complete straight line.

The \emph{dual cone} of $\Omega$ is defined by
\[
\Omega^\ast
=
\{\xi\in\mathbb{R}^d:\langle x,\xi\rangle \ge 0
\ \text{for all }x\in\Omega\}.
\]
Thus $\Omega^\ast$ is a closed convex cone in $\mathbb{R}^d$.

Associated with $\Omega$ we consider the tube domain
\[
T_\Omega
=
\mathbb{R}^d+i\Omega
=
\{z=x+iy:\;x\in\mathbb{R}^d,\;y\in\Omega\}
\subset\mathbb{C}^d ,
\]
where $\mathbb{C}^d$ denotes the $d$-dimensional complex Euclidean space.

Similarly we denote
\[
T_{-\Omega}
=
\mathbb{R}^d-i\Omega .
\]

For $y\in\Omega$, the horizontal slice
\[
\mathbb{R}^d+iy
\]
may be identified with $\mathbb{R}^d$ through the real coordinate
$x\mapsto x+iy$. This identification will be used when discussing
boundary values of analytic functions on $T_\Omega$.

\subsection{Hardy spaces on tube domains}

We briefly recall the Hardy space on $T_\Omega$ in the Hilbertian case.

Let $F$ be holomorphic on $T_\Omega$. We say that $F$ belongs to the
Hardy space $H^2(T_\Omega)$ if
\[
\sup_{y\in\Omega}
\int_{\mathbb{R}^d}
|F(x+iy)|^2\,dx
<\infty .
\]
The quantity
\[
\|F\|_{H^2(T_\Omega)}
=
\sup_{y\in\Omega}
\left(
\int_{\mathbb{R}^d}|F(x+iy)|^2dx
\right)^{1/2}
\]
defines a norm under which $H^2(T_\Omega)$ becomes a Hilbert space.

A fundamental feature of Hardy spaces on tube domains is that they
admit Fourier representations supported in the dual cone.
More precisely, if $F\in H^2(T_\Omega)$, then there exists a function
$\widehat f\in L^2(\Omega^\ast)$ such that
\[
F(z)
=
\int_{\Omega^\ast}
e^{i\langle z,\xi\rangle}\widehat f(\xi)\,d\xi,
\qquad z\in T_\Omega,
\]
and
\[
\|F\|_{H^2(T_\Omega)}^2
=
\int_{\Omega^\ast}
|\widehat f(\xi)|^2\,d\xi .
\]
Conversely, every function defined by the above integral with
$\widehat f\in L^2(\Omega^\ast)$ belongs to $H^2(T_\Omega)$.
Such Fourier--Laplace representations play a central role in the
analysis of Hardy spaces on tube domains; see, for instance,
\cite{Garrigos2001,LiDengQian2018,DengLiQian2019,SteinWeiss1971}.

\subsection{Fourier transform and Fourier--Laplace representations}

For $f\in L^1(\mathbb{R}^d)$ we define the Fourier transform by
\[
\widehat f(\xi)
=
\int_{\mathbb{R}^d}
f(x)e^{-i\langle x,\xi\rangle}\,dx,
\qquad \xi\in\mathbb{R}^d.
\]
This definition extends in the usual way to $L^2(\mathbb{R}^d)$, where
the Fourier transform becomes a unitary operator on $L^2(\mathbb{R}^d)$;
see \cite{SteinWeiss1971}.

If a function $\widehat f$ is supported in the dual cone $\Omega^\ast$,
the associated Fourier integral naturally extends to a holomorphic
function on $T_\Omega$ through the formula
\[
F(z)
=
\int_{\Omega^\ast}
e^{i\langle z,\xi\rangle}\widehat f(\xi)\,d\xi,
\qquad z\in T_\Omega.
\]
This integral representation is usually referred to as the
\emph{Fourier--Laplace transform} associated with the cone $\Omega$.
It provides the analytic continuation of boundary Fourier data into
the tube domain and underlies Paley--Wiener type characterizations of
Hardy spaces on tubes.

\subsection{Sobolev spaces on Euclidean space}

We briefly recall the classical Sobolev spaces on $\mathbb{R}^d$.
For $n\in\mathbb{N}$ the Sobolev space $W^n_2(\mathbb{R}^d)$ consists
of all functions $u\in L^2(\mathbb{R}^d)$ whose weak derivatives
$\partial^\alpha u$ of order $|\alpha|\le n$ belong to $L^2(\mathbb{R}^d)$.
It is equipped with the norm
\[
\|u\|_{W^n_2(\mathbb{R}^d)}^2
=
\sum_{|\alpha|\le n}
\|\partial^\alpha u\|_{L^2(\mathbb{R}^d)}^2 .
\]

In the Fourier domain, this norm can be expressed through polynomial
weights. More precisely, one has
\[
\|u\|_{W^n_2(\mathbb{R}^d)}^2
\asymp
\int_{\mathbb{R}^d}
|\widehat u(\xi)|^2
(1+|\xi|^2)^n
\,d\xi ,
\]
where $\widehat u$ denotes the Fourier transform of $u$.
For further background on Sobolev spaces we refer to
\cite{AdamsFournier2003,Mazya2011,Triebel1983}.

To formulate the Hardy--Sobolev space studied in this paper, we shall
later fix a cone-adapted positive homogeneous gauge on the dual cone
$\Omega^\ast$ and use it to define weighted Fourier-side norms.

\section{Hilbertian Hardy--Sobolev Spaces on Tube}

In this section we introduce the weighted Hilbertian definition of
Hardy--Sobolev spaces on tube domains over convex cones that will be
studied throughout the paper.
As discussed in the Introduction, Hardy--Sobolev spaces admit several
natural definitions in the literature.
For the purposes of the present paper, we adopt a weighted Hilbertian
definition on the Fourier--Laplace side.
This definition is adapted to the geometry of the dual cone and is
particularly well suited to the Paley--Wiener analysis developed in the
next section.

\subsection{A weighted Fourier-side definition}

To define the Hilbertian Hardy--Sobolev space on $T_\Omega$ in a way
adapted to the cone geometry, we introduce a weighted Fourier-side
framework on the dual cone $\Omega^\ast$.

Throughout the paper we fix a continuous positive homogeneous function
\[
\rho:\Omega^\ast\to(0,\infty)
\]
of degree one, that is,
\[
\rho(t\xi)=t\rho(\xi),
\qquad t>0,\ \xi\in\Omega^\ast.
\]
This function plays the role of a cone-adapted frequency gauge on
$\Omega^\ast$.

For $n\in\mathbb N$, we define the associated weight
\[
w_n(\xi)=\sum_{k=0}^n \rho(\xi)^{2k},
\qquad \xi\in\Omega^\ast.
\]

The precise choice of $\rho$ is not essential at the level of the present
development; what matters is positivity, continuity, and homogeneity of
degree one. Different admissible choices of $\rho$ lead to equivalent
weighted descriptions whenever the corresponding gauges are comparable.
In this paper, however, one such admissible gauge is fixed once and for all.

\begin{remark}
Since $\rho(\xi)>0$ on $\Omega^\ast$, we have
\[
w_n(\xi)\asymp (1+\rho(\xi)^2)^n,
\qquad \xi\in\Omega^\ast.
\]
Thus either expression may be used at the level of equivalent norms.
We work with $w_n$ in the form $\sum_{k=0}^n \rho(\xi)^{2k}$ because this
notation is convenient for later estimates involving different orders.
\end{remark}

For measurable functions $f$ on $\Omega^\ast$, we write
\[
\|f\|_{L^2(\Omega^\ast,w_n)}^2
=
\int_{\Omega^\ast}|f(\xi)|^2w_n(\xi)\,d\xi.
\]

Motivated by the Fourier--Laplace description of Hardy spaces on tube
domains and by the Sobolev principle of measuring smoothness through
weighted Fourier integrability, we introduce the following Hilbertian
Hardy--Sobolev space. The definition combines the Hardy-space spectral
support condition in the dual cone $\Omega^\ast$ with the Sobolev-type
weight $w_n$ on the Fourier side.

\begin{definition}
Let $n\in\mathbb{N}$.
The \emph{Hilbertian Hardy--Sobolev space} $H^n_{2,\rho}(T_\Omega)$ consists
of all holomorphic functions $F$ on $T_\Omega$ of the form
\[
F(z)=\int_{\Omega^\ast} e^{i\langle z,\xi\rangle}f(\xi)\,d\xi,
\qquad z\in T_\Omega,
\]
for some $f\in L^2(\Omega^\ast,w_n)$.
The norm on $H^n_{2,\rho}(T_\Omega)$ is defined by
\[
\|F\|_{H^n_{2,\rho}(T_\Omega)}
=
\inf\bigl\{\|f\|_{L^2(\Omega^\ast,w_n)}:F=\mathcal L f\bigr\},
\]
where
\[
\mathcal L f(z)=\int_{\Omega^\ast} e^{i\langle z,\xi\rangle}f(\xi)\,d\xi
\]
denotes the Fourier--Laplace transform.
\end{definition}

The usefulness of this definition will be justified in the next section,
where we show that it admits an intrinsic Paley--Wiener
characterization in terms of boundary translates and weighted spectral
densities. This characterization will provide the foundation for the
subsequent decomposition, reproducing-kernel and Carleson-measure
results.

\begin{remark}
When $d=1$ and $\Omega=(0,\infty)$, we may choose $\rho(\xi)=\xi$ on
$\Omega^\ast=(0,\infty)$. Then the present definition reduces to the
Fourier--Laplace image of the weighted space
\[
L^2\!\left((0,\infty),\sum_{k=0}^n \xi^{2k}\,d\xi\right).
\]
Hence, in one dimension, our space is a weighted Hilbertian
Hardy--Sobolev space on the upper half-plane. It is therefore closely
related to the weighted Hilbertian theory studied in
\cite{LiuLiKouQu2025}. We do not pursue here the precise equivalence
between the two formulations.
\end{remark}

\subsection{Well-definedness and Hilbert structure}

We next verify that the Fourier--Laplace definition above is well posed.
More precisely, we show that the resulting space is exactly the image of
the weighted space $L^2(\Omega^\ast,w_n)$ under the Fourier--Laplace
transform, and that it inherits from that space a natural Hilbert
structure.

\begin{theorem}\label{thm:HS-Fourier-Hilbert}
For every $n\in\mathbb{N}$, the Fourier--Laplace transform
\[
\mathcal L:L^2(\Omega^\ast,w_n)\to H(T_\Omega)
\]
is injective, and its range is precisely $H^n_{2,\rho}(T_\Omega)$.
Consequently, the formula
\[
\langle \mathcal L f,\mathcal L g\rangle_{H^n_{2,\rho}(T_\Omega)}
:=
\int_{\Omega^\ast} f(\xi)\overline{g(\xi)}\,w_n(\xi)\,d\xi
\]
defines an inner product on $H^n_{2,\rho}(T_\Omega)$, with respect to
which $H^n_{2,\rho}(T_\Omega)$ is a Hilbert space.
In particular, $\mathcal L$ is an isometric isomorphism from
$L^2(\Omega^\ast,w_n)$ onto $H^n_{2,\rho}(T_\Omega)$.
\end{theorem}

Thus $H^n_{2,\rho}(T_\Omega)$ may be regarded as the holomorphic
realization of the weighted Hilbert space $L^2(\Omega^\ast,w_n)$ under
the Fourier--Laplace transform.

\begin{proof}
Let $f\in L^2(\Omega^\ast,w_n)$. Since
\[
w_n(\xi)=\sum_{k=0}^n \rho(\xi)^{2k}\ge 1,
\]
we have $f\in L^2(\Omega^\ast)$.

Fix $z=x+iy\in T_\Omega$. Since $y\in\Omega$, there exists a constant
$c_y>0$ such that
\[
\langle y,\xi\rangle \ge c_y |\xi|,
\qquad \xi\in \Omega^\ast.
\]
Hence, by Cauchy--Schwarz,
\[
\int_{\Omega^\ast}\bigl|e^{i\langle z,\xi\rangle}f(\xi)\bigr|\,d\xi
=
\int_{\Omega^\ast} e^{-\langle y,\xi\rangle}|f(\xi)|\,d\xi
\]
\[
\le
\left(\int_{\Omega^\ast}|f(\xi)|^2w_n(\xi)\,d\xi\right)^{1/2}
\left(\int_{\Omega^\ast}e^{-2\langle y,\xi\rangle}w_n(\xi)^{-1}\,d\xi\right)^{1/2}
\]
and
\[
\int_{\Omega^\ast}e^{-2\langle y,\xi\rangle}w_n(\xi)^{-1}\,d\xi
\le
\int_{\Omega^\ast}e^{-2c_y|\xi|}\,d\xi<\infty.
\]
Therefore the integral defining $\mathcal Lf(z)$ converges absolutely for
every $z\in T_\Omega$.

We next show that $\mathcal Lf$ is holomorphic on $T_\Omega$. Let
$K\subset T_\Omega$ be compact. Since $\Im z\in\Omega$ for all $z\in K$,
there exists $c_K>0$ such that
\[
\langle \Im z,\xi\rangle \ge c_K|\xi|,
\qquad z\in K,\ \xi\in\Omega^\ast.
\]
Thus, for each $j=1,\dots,d$,
\[
\bigl|\xi_j e^{i\langle z,\xi\rangle}f(\xi)\bigr|
\le
|\xi| e^{-c_K|\xi|}|f(\xi)|.
\]
By Cauchy--Schwarz,
\[
\int_{\Omega^\ast} |\xi| e^{-c_K|\xi|}|f(\xi)|\,d\xi
\le
\|f\|_{L^2(\Omega^\ast,w_n)}
\left(\int_{\Omega^\ast} |\xi|^2 e^{-2c_K|\xi|}w_n(\xi)^{-1}\,d\xi\right)^{1/2}
<\infty.
\]
Hence differentiation under the integral sign is justified on $K$, and
\[
\frac{\partial}{\partial z_j}\mathcal Lf(z)
=
i\int_{\Omega^\ast}\xi_j e^{i\langle z,\xi\rangle}f(\xi)\,d\xi,
\qquad z\in T_\Omega.
\]
Therefore $\mathcal Lf$ is holomorphic on $T_\Omega$, so
\[
\mathcal L\bigl(L^2(\Omega^\ast,w_n)\bigr)\subset H(T_\Omega).
\]

By the definition of $H^n_{2,\rho}(T_\Omega)$, the range of $\mathcal L$
is precisely $H^n_{2,\rho}(T_\Omega)$.

To prove injectivity, suppose that $\mathcal Lf=0$ on $T_\Omega$. Fix
$y\in\Omega$ and define
\[
g_y(\xi):=e^{-\langle y,\xi\rangle}f(\xi)\mathbf \chi_{\Omega^\ast}(\xi).
\]
Then $g_y\in L^1(\mathbb R^d)$ by the above estimate, and for every
$x\in\mathbb R^d$,
\[
\widehat{g_y}(-x)=\mathcal Lf(x+iy)=0.
\]
By the uniqueness theorem for the Fourier transform, $g_y=0$ almost
everywhere on $\mathbb R^d$. Since $e^{-\langle y,\xi\rangle}>0$ on
$\Omega^\ast$, it follows that $f=0$ almost everywhere on $\Omega^\ast$.
Thus $\mathcal L$ is injective.

Finally, define
\[
\langle \mathcal Lf,\mathcal Lg\rangle_{H^n_{2,\rho}(T_\Omega)}
:=
\int_{\Omega^\ast} f(\xi)\overline{g(\xi)}\,w_n(\xi)\,d\xi.
\]
This is well defined because $\mathcal L$ is injective. By construction,
$\mathcal L$ is an isometry from $L^2(\Omega^\ast,w_n)$ onto
$H^n_{2,\rho}(T_\Omega)$, and since $L^2(\Omega^\ast,w_n)$ is complete,
$H^n_{2,\rho}(T_\Omega)$ is a Hilbert space.
\end{proof}

At this stage the space is defined on the Fourier--Laplace side.
Its intrinsic analytic interpretation is still hidden.
The next section will show that this definition admits a Paley--Wiener
type characterization in terms of boundary translates and spectral
support in the dual cone.

\subsection{Basic analytic properties}

We  next establish several basic analytic properties of
$H^n_{2,\rho}(T_\Omega)$ that follow directly from its Fourier--Laplace
realization as a weighted Hilbert space.

\begin{proposition}\label{prop:HS-point-evaluation}
For each $z\in T_\Omega$, the point evaluation functional
\[
\delta_z:F\mapsto F(z)
\]
is bounded on $H^n_{2,\rho}(T_\Omega)$.
Consequently, $H^n_{2,\rho}(T_\Omega)$ is a reproducing kernel Hilbert
space.
\end{proposition}

\begin{proof}
Let $F=\mathcal L f\in H^n_{2,\rho}(T_\Omega)$, where
$f\in L^2(\Omega^\ast,w_n)$. For $z=x+iy\in T_\Omega$, the estimate from
the proof of Theorem~\ref{thm:HS-Fourier-Hilbert} gives
\[
|F(z)|
\le
\|f\|_{L^2(\Omega^\ast,w_n)}
\left(
\int_{\Omega^\ast}
e^{-2\langle y,\xi\rangle}w_n(\xi)^{-1}\,d\xi
\right)^{1/2}.
\]
Since
\[
\|f\|_{L^2(\Omega^\ast,w_n)}
=
\|F\|_{H^n_{2,\rho}(T_\Omega)},
\]
it follows that for each fixed $z\in T_\Omega$, there exists a constant
$C_z>0$ such that
\[
|F(z)|\le C_z \|F\|_{H^n_{2,\rho}(T_\Omega)},
\qquad F\in H^n_{2,\rho}(T_\Omega).
\]
Hence the point evaluation functional $\delta_z$ is bounded.
Since $H^n_{2,\rho}(T_\Omega)$ is a Hilbert space, the reproducing-kernel
property follows from the Riesz representation theorem.
\end{proof}

\begin{proposition}\label{prop:HS-inclusion-rho}
Let $m,n\in\mathbb{N}$ with $m\le n$.
Then
\[
H^n_{2,\rho}(T_\Omega)\subset H^m_{2,\rho}(T_\Omega)
\]
continuously. More precisely,
\[
\|F\|_{H^m_{2,\rho}(T_\Omega)}
\le
\|F\|_{H^n_{2,\rho}(T_\Omega)},
\qquad
F\in H^n_{2,\rho}(T_\Omega).
\]
\end{proposition}

\begin{proof}
Let $F=\mathcal L f\in H^n_{2,\rho}(T_\Omega)$, where
$f\in L^2(\Omega^\ast,w_n)$. Since
\[
w_m(\xi)=\sum_{k=0}^m \rho(\xi)^{2k}
\le
\sum_{k=0}^n \rho(\xi)^{2k}
=
w_n(\xi),
\]
we have $f\in L^2(\Omega^\ast,w_m)$ and
\[
\|F\|_{H^m_{2,\rho}(T_\Omega)}
=
\|f\|_{L^2(\Omega^\ast,w_m)}
\le
\|f\|_{L^2(\Omega^\ast,w_n)}
=
\|F\|_{H^n_{2,\rho}(T_\Omega)}.
\]
This proves the claimed continuous inclusion.
\end{proof}

\begin{corollary}\label{cor:HS-local-uniform}
Let $\{F_j\}$ be a sequence in $H^n_{2,\rho}(T_\Omega)$ converging to
$F$ in $H^n_{2,\rho}(T_\Omega)$.
Then
\[
F_j\to F
\]
locally uniformly on $T_\Omega$.
\end{corollary}

\begin{proof}
Let $K\subset T_\Omega$ be compact. Since $\Im z\in \Omega$ for every
$z\in K$, there exists a constant $c_K>0$ such that
\[
\langle \Im z,\xi\rangle \ge c_K |\xi|,
\qquad z\in K,\ \xi\in\Omega^\ast.
\]
If $G=\mathcal L g\in H^n_{2,\rho}(T_\Omega)$, then for every $z\in K$,
\[
|G(z)|
\le
\|g\|_{L^2(\Omega^\ast,w_n)}
\left(
\int_{\Omega^\ast}
e^{-2\langle \Im z,\xi\rangle}w_n(\xi)^{-1}\,d\xi
\right)^{1/2}
\]
\[
\le
\|G\|_{H^n_{2,\rho}(T_\Omega)}
\left(
\int_{\Omega^\ast}
e^{-2c_K|\xi|}w_n(\xi)^{-1}\,d\xi
\right)^{1/2}.
\]
Hence there exists a constant $C_K>0$ such that
\[
\sup_{z\in K}|G(z)|
\le
C_K\|G\|_{H^n_{2,\rho}(T_\Omega)}.
\]
Applying this estimate to $G=F_j-F$, we obtain
\[
\sup_{z\in K}|F_j(z)-F(z)|
\le
C_K\|F_j-F\|_{H^n_{2,\rho}(T_\Omega)}
\to 0.
\]
Therefore $F_j\to F$ uniformly on $K$. Since $K$ was arbitrary,
$F_j\to F$ locally uniformly on $T_\Omega$.
\end{proof}

The weighted Fourier-side definition adopted above already yields a
well-behaved Hilbert space of holomorphic functions with bounded point
evaluations. What remains to be clarified is the analytic content of the
norm from the viewpoint of derivative control and Paley--Wiener theory.
The next section addresses this by identifying the corresponding
Fourier--Laplace/Paley--Wiener structure and relating the norm to
analytic derivatives.

\section{Paley--Wiener Characterization}

In Section~3, the Hilbertian Hardy--Sobolev space
$H^{n}_{2,\rho}(T_\Omega)$ was introduced through a weighted
Fourier--Laplace representation on the dual cone $\Omega^\ast$.
We now show that this definition admits an intrinsic interpretation
in terms of boundary translates and weighted spectral characterization.
More precisely, we establish a Paley--Wiener type characterization
which identifies $H^{n}_{2,\rho}(T_\Omega)$ as the class of holomorphic
functions on $T_\Omega$ whose boundary translates have Fourier spectra
supported in $\Omega^\ast$ and governed by a common weighted $L^2$ density.
As a further consequence, the weighted Fourier-side norm yields the
corresponding control of analytic derivatives.

\subsection{Boundary translates}

This observation provides the forward implication in the Paley--Wiener
characterization and serves as the bridge between the Fourier--Laplace
definition of $H^n_{2,\rho}(T_\Omega)$ and its intrinsic description in
terms of boundary translates.

For $F\in H(T_\Omega)$ and $y\in\Omega$, we define
\[
F_y(x):=F(x+iy), \qquad x\in\mathbb{R}^d.
\]
Thus $F_y$ denotes the translate of $F$ at height $y$ parallel to the
boundary of $T_\Omega$.

We begin by describing the Fourier transforms of the translates of a
Fourier--Laplace transform.

\begin{proposition}\label{prop4.1}
Let $n\in\mathbb{N}$, let $f\in L^2(\Omega^\ast,w_n)$, and define
\[
F=\mathcal{L}f,
\qquad
F(z)=\int_{\Omega^\ast} e^{i\langle z,\xi\rangle}f(\xi)\,d\xi,
\qquad z\in T_\Omega.
\]
Then for every $y\in\Omega$, the function $F_y$ belongs to
$L^2(\mathbb{R}^d)$ and satisfies
\[
\widehat{F_y}(\xi)
=
e^{-\langle y,\xi\rangle}f(\xi)\chi_{\Omega^\ast}(\xi)
\qquad \text{for a.e. } \xi\in\mathbb{R}^d.
\]
Moreover,
\[
\|F_y\|_{L^2(\mathbb{R}^d)}^2
\asymp
\int_{\Omega^\ast} e^{-2\langle y,\xi\rangle}|f(\xi)|^2\,d\xi,
\]
and in particular
\[
\sup_{y\in\Omega}\|F_y\|_{L^2(\mathbb{R}^d)}
\lesssim
\|f\|_{L^2(\Omega^\ast,w_n)}.
\]
Hence $F\in H^2(T_\Omega)$.
\end{proposition}

\begin{proof}
Fix $y\in\Omega$ and define
\[
g_y(\xi):=e^{-\langle y,\xi\rangle}f(\xi)\chi_{\Omega^\ast}(\xi),
\qquad \xi\in\mathbb{R}^d.
\]
Since $f\in L^2(\Omega^\ast,w_n)$ and $w_n\ge 1$, we have
$f\in L^2(\Omega^\ast)$. Moreover, for $\xi\in\Omega^\ast$,
\[
0<e^{-\langle y,\xi\rangle}\le 1,
\]
and therefore
\[
|g_y(\xi)|^2
=
e^{-2\langle y,\xi\rangle}|f(\xi)|^2\chi_{\Omega^\ast}(\xi)
\le
|f(\xi)|^2\chi_{\Omega^\ast}(\xi).
\]
It follows that $g_y\in L^2(\mathbb{R}^d)$.

For $x\in\mathbb{R}^d$, we have
\[
F_y(x)
=
F(x+iy)
=
\int_{\Omega^\ast} e^{i\langle x,\xi\rangle}e^{-\langle y,\xi\rangle}f(\xi)\,d\xi
=
\int_{\mathbb{R}^d} e^{i\langle x,\xi\rangle}g_y(\xi)\,d\xi.
\]
Thus $F_y$ is the inverse Fourier transform of $g_y$. By Plancherel's
theorem, $F_y\in L^2(\mathbb{R}^d)$ and
\[
\widehat{F_y}(\xi)=g_y(\xi)
=
e^{-\langle y,\xi\rangle}f(\xi)\chi_{\Omega^\ast}(\xi)
\]
for almost every $\xi\in\mathbb{R}^d$. In particular,
\[
\|F_y\|_{L^2(\mathbb{R}^d)}^2
\asymp
\|g_y\|_{L^2(\mathbb{R}^d)}^2
=
\int_{\Omega^\ast} e^{-2\langle y,\xi\rangle}|f(\xi)|^2\,d\xi.
\]
Since $e^{-2\langle y,\xi\rangle}\le 1$ on $\Omega^\ast$, we further obtain
\[
\sup_{y\in\Omega}\|F_y\|_{L^2(\mathbb{R}^d)}^2
\lesssim
\int_{\Omega^\ast}|f(\xi)|^2\,d\xi
\le
\int_{\Omega^\ast}|f(\xi)|^2w_n(\xi)\,d\xi.
\]
Therefore $F\in H^2(T_\Omega)$.
\end{proof}

At this level, the conclusion is still formally analogous to the
classical $H^2$ case: the Hardy--Sobolev feature is not merely the
support condition in $\Omega^\ast$, but the additional weighted
$L^2(\Omega^\ast,w_n)$ control of the common spectral density.
We next show that these two ingredients together characterize the space.

\subsection{Paley--Wiener characterization}

This subsection contains the main result of the section. While
$H^{n}_{2,\rho}(T_\Omega)$ was introduced in Section~3 as the
Fourier--Laplace image of a weighted space on the dual cone $\Omega^\ast$,
the theorem below shows that it can also be characterized intrinsically
in terms of its translates. In this sense, it provides the precise
Paley--Wiener bridge between the Fourier-side definition and the
function-theoretic description of the space.

\begin{theorem}[Paley--Wiener characterization]\label{thm4.2}
Let $n\in\mathbb{N}$ and let $F\in H(T_\Omega)$. Then the following
statements are equivalent.

\begin{enumerate}
\item[(i)] $F\in H^{n}_{2,\rho}(T_\Omega)$.

\item[(ii)] There exists a unique function $f\in L^2(\Omega^\ast,w_n)$
such that
\[
F(z)
=
\int_{\Omega^\ast} e^{i\langle z,\xi\rangle}f(\xi)\,d\xi,
\qquad z\in T_\Omega.
\]

\item[(iii)] There exists a unique function $f\in L^2(\Omega^\ast,w_n)$
such that, for every $y\in\Omega$, one has $F_y\in L^2(\mathbb{R}^d)$ and
\[
\widehat{F_y}(\xi)
=
e^{-\langle y,\xi\rangle}f(\xi)\chi_{\Omega^\ast}(\xi)
\qquad \text{for a.e. } \xi\in\mathbb{R}^d.
\]
\end{enumerate}

Moreover, whenever these conditions hold, one has
\[
\|F\|_{H^{n}_{2,\rho}(T_\Omega)}
=
\|f\|_{L^2(\Omega^\ast,w_n)}.
\]
\end{theorem}

\begin{proof}
The equivalence of \textup{(i)} and \textup{(ii)} is immediate from the
definition of $H^{n}_{2,\rho}(T_\Omega)$ and the injectivity of the
Fourier--Laplace transform established in
Theorem~\ref{thm:HS-Fourier-Hilbert}.

The implication \textup{(ii)}$\Rightarrow$\textup{(iii)} follows from
Proposition~\ref{prop4.1}.

It remains to prove \textup{(iii)}$\Rightarrow$\textup{(ii)}.
Assume that \textup{(iii)} holds, and define
\[
G(z):=\int_{\Omega^\ast} e^{i\langle z,\xi\rangle}f(\xi)\,d\xi,
\qquad z\in T_\Omega.
\]
Since $f\in L^2(\Omega^\ast,w_n)$, we have
$G=\mathcal{L}f\in H^{n}_{2,\rho}(T_\Omega)$ by
Theorem~\ref{thm:HS-Fourier-Hilbert}.

Fix $y\in\Omega$. By assumption,
\[
\widehat{F_y}(\xi)
=
e^{-\langle y,\xi\rangle}f(\xi)\chi_{\Omega^\ast}(\xi)
\qquad \text{for a.e. } \xi\in\mathbb{R}^d.
\]
On the other hand, Proposition~\ref{prop4.1} applied to $G=\mathcal{L}f$
shows that
\[
\widehat{G_y}(\xi)
=
e^{-\langle y,\xi\rangle}f(\xi)\chi_{\Omega^\ast}(\xi)
\qquad \text{for a.e. } \xi\in\mathbb{R}^d.
\]
Hence
\[
\widehat{F_y}=\widehat{G_y}
\qquad \text{a.e. on }\mathbb{R}^d.
\]
By the uniqueness of the Fourier transform in $L^2(\mathbb{R}^d)$, it follows that
\[
F_y=G_y
\qquad \text{a.e. on }\mathbb{R}^d.
\]
Since both $x\mapsto F(x+iy)$ and $x\mapsto G(x+iy)$ are continuous on
$\mathbb{R}^d$, we conclude that
\[
F(x+iy)=G(x+iy), \qquad x\in\mathbb{R}^d.
\]
As $y\in\Omega$ was arbitrary, it follows that
\[
F(z)=G(z)=\mathcal{L}f(z), \qquad z\in T_\Omega.
\]
Thus \textup{(ii)} holds.

The uniqueness of $f$ follows from the injectivity of the
Fourier--Laplace transform. Finally, the norm identity
\[
\|F\|_{H^{n}_{2,\rho}(T_\Omega)}
=
\|f\|_{L^2(\Omega^\ast,w_n)}
\]
is precisely the transported norm introduced in
Theorem~\ref{thm:HS-Fourier-Hilbert}.
\end{proof}

The content of Theorem~\ref{thm4.2} is not merely that
$H^{n}_{2,\rho}(T_\Omega)$ is defined as the range of a weighted
Fourier--Laplace transform. Rather, it shows that this space can be
recognized intrinsically among holomorphic functions on $T_\Omega$ by
means of translates whose Fourier transforms are supported in
$\Omega^\ast$ and are governed by a common spectral density satisfying
the weighted $L^2$ condition determined by $w_n$.

\subsection{Analytic derivative estimates}

We now explain how the weighted spectral control obtained from the
Paley--Wiener characterization yields the Sobolev component of the
theory. More precisely, the weighted condition on the common spectral
density gives boundedness for polynomial Fourier multipliers, and under a
natural compatibility assumption between the cone gauge $\rho$ and the
Euclidean frequency, it yields Hardy-type estimates for analytic
derivatives up to order $n$. This is the point at which the present
framework departs from the classical $H^2$ theory and exhibits its
genuinely Hardy--Sobolev character.

For a polynomial
\[
P(\zeta)=\sum_{|\alpha|\le m} a_\alpha \zeta^\alpha
\qquad (\zeta\in\mathbb C^d),
\]
we write
\[
P(D):=\sum_{|\alpha|\le m} a_\alpha \partial_z^\alpha,
\qquad
D=(\partial_{z_1},\dots,\partial_{z_d}),
\]
where $\partial_z^\alpha=\partial_{z_1}^{\alpha_1}\cdots\partial_{z_d}^{\alpha_d}$.

\begin{proposition}\label{prop4.3}
Let $F=\mathcal{L}f\in H^{n}_{2,\rho}(T_\Omega)$, and let $P$ be a polynomial
of degree at most $n$ such that
\[
|P(i\xi)|\le C_P\, w_n(\xi)^{1/2},
\qquad \xi\in\Omega^\ast.
\]
Then
\[
P(D)F\in H^2(T_\Omega),
\]
and
\[
\|P(D)F\|_{H^2(T_\Omega)}
\lesssim
C_P\,\|F\|_{H^{n}_{2,\rho}(T_\Omega)}.
\]
More precisely,
\[
P(D)F
=
\mathcal{L}(P(i\xi)f(\xi)).
\]
\end{proposition}

\begin{proof}
Set
\[
g(\xi):=P(i\xi)f(\xi), \qquad \xi\in\Omega^\ast.
\]
By assumption,
\[
|g(\xi)|^2
=
|P(i\xi)|^2 |f(\xi)|^2
\le
C_P^2\, w_n(\xi)\,|f(\xi)|^2.
\]
Hence
\[
\int_{\Omega^\ast}|g(\xi)|^2\,d\xi
\le
C_P^2
\int_{\Omega^\ast}|f(\xi)|^2w_n(\xi)\,d\xi
<
\infty,
\]
so $g\in L^2(\Omega^\ast)$.

Define
\[
G(z):=\int_{\Omega^\ast} e^{i\langle z,\xi\rangle}g(\xi)\,d\xi,
\qquad z\in T_\Omega.
\]
By the classical Paley--Wiener theorem for $H^2(T_\Omega)$ recalled in
Section~2, it follows that $G\in H^2(T_\Omega)$ and
\[
\|G\|_{H^2(T_\Omega)}
\asymp
\|g\|_{L^2(\Omega^\ast)}
\le
C_P\,\|f\|_{L^2(\Omega^\ast,w_n)}.
\]
Since
\[
\|f\|_{L^2(\Omega^\ast,w_n)}
=
\|F\|_{H^{n}_{2,\rho}(T_\Omega)},
\]
we obtain
\[
\|G\|_{H^2(T_\Omega)}
\lesssim
C_P\,\|F\|_{H^{n}_{2,\rho}(T_\Omega)}.
\]

It remains to identify $G$ with $P(D)F$. Write
\[
P(\zeta)=\sum_{|\alpha|\le m} a_\alpha \zeta^\alpha,
\qquad m\le n.
\]
Let $K\subset T_\Omega$ be compact. Then there exists $c_K>0$ such that
\[
\langle \Im z,\xi\rangle \ge c_K |\xi|,
\qquad z\in K,\ \xi\in\Omega^\ast.
\]
For each multi-index $\alpha$ with $|\alpha|\le m$, we have
\[
\int_{\Omega^\ast}
|\xi^\alpha e^{i\langle z,\xi\rangle}f(\xi)|\,d\xi
\le
\int_{\Omega^\ast}
|\xi|^{|\alpha|}e^{-\langle \Im z,\xi\rangle}|f(\xi)|\,d\xi
\]
\[
\le
\|f\|_{L^2(\Omega^\ast)}
\left(
\int_{\Omega^\ast}
|\xi|^{2|\alpha|}e^{-2c_K|\xi|}\,d\xi
\right)^{1/2}
<
\infty,
\qquad z\in K,
\]
by Cauchy--Schwarz. Thus differentiation under the integral sign is
justified on compact subsets of $T_\Omega$, and for every $z\in T_\Omega$,
\[
\partial_z^\alpha F(z)
=
\int_{\Omega^\ast} e^{i\langle z,\xi\rangle}(i\xi)^\alpha f(\xi)\,d\xi.
\]
Summing over $\alpha$, we obtain
\[
P(D)F(z)
=
\int_{\Omega^\ast} e^{i\langle z,\xi\rangle}P(i\xi)f(\xi)\,d\xi
=
G(z)
=
\mathcal{L}(P(i\xi)f(\xi))(z).
\]
This proves the result.
\end{proof}

The preceding proposition shows that the weighted norm defining
$H^{n}_{2,\rho}(T_\Omega)$ controls precisely those analytic differential
operators whose Fourier multipliers are dominated by the
Hardy--Sobolev weight.

A particularly transparent derivative formulation is obtained when the
cone-adapted gauge $\rho$ dominates the Euclidean frequency.

\begin{corollary}\label{cor4.4}
Assume that there exists a constant $C_0>0$ such that
\[
|\xi|\le C_0\,\rho(\xi), \qquad \xi\in\Omega^\ast.
\]
Then for every multi-index $\alpha$ with $|\alpha|\le n$ one has
\[
\partial^\alpha F\in H^2(T_\Omega)
\qquad \text{for all } F\in H^{n}_{2,\rho}(T_\Omega),
\]
and
\[
\|\partial^\alpha F\|_{H^2(T_\Omega)}
\le
C_\alpha\,\|F\|_{H^{n}_{2,\rho}(T_\Omega)}.
\]
\end{corollary}

\begin{proof}
Let $|\alpha|\le n$, and consider the polynomial
\[
P(\zeta):=\zeta^\alpha.
\]
Then for $\xi\in\Omega^\ast$,
\[
|P(i\xi)|=|(i\xi)^\alpha|=|\xi^\alpha|
\le
|\xi|^{|\alpha|}
\le
C_0^{|\alpha|}\rho(\xi)^{|\alpha|}.
\]
Since
\[
w_n(\xi)=\sum_{k=0}^{n}\rho(\xi)^{2k},
\]
the term $\rho(\xi)^{2|\alpha|}$ appears in $w_n(\xi)$, and therefore
\[
|\xi^\alpha|
\le
C_\alpha\,w_n(\xi)^{1/2}.
\]
Applying Proposition~\ref{prop4.3} with $P(\zeta)=\zeta^\alpha$ yields
\[
\partial^\alpha F=P(D)F\in H^2(T_\Omega)
\]
and
\[
\|\partial^\alpha F\|_{H^2(T_\Omega)}
\le
C_\alpha\,\|F\|_{H^{n}_{2,\rho}(T_\Omega)}.
\]
\end{proof}

\begin{remark}
Corollary~\ref{cor4.4} shows that, under a natural compatibility
assumption between the cone gauge $\rho$ and the Euclidean frequency,
the present weighted Hilbertian definition recovers the familiar
Hardy--Sobolev principle that analytic derivatives up to order $n$
belong to the underlying Hardy space. In this sense, the weighted
Fourier--Laplace model introduced in Section~3 is fully consistent with
the classical derivative-based intuition behind Hardy--Sobolev spaces.
\end{remark}

The Paley--Wiener characterization and analytic derivative estimates
established in this section identify the Fourier-side structure of
$H^n_{2,\rho}(T_\Omega)$ and its associated boundary Sobolev regularity.
These results naturally lead to a boundary decomposition problem:
namely, how Sobolev functions whose Fourier spectrum is supported in
$\Omega^\ast\cup(-\Omega^\ast)$ split into Hardy--Sobolev components
associated with the opposite tube domains $T_\Omega$ and $T_{-\Omega}$.
This will be the subject of the next section.

\section{A Hardy--Sobolev decomposition}

The Paley--Wiener characterization established in Section~4 identifies
$H^n_{2,\rho}(T_\Omega)$ with weighted Fourier spectra supported in
$\Omega^\ast$. This naturally leads to the question of how the
corresponding boundary Sobolev functions decompose relative to the two
opposite tube domains $T_\Omega$ and $T_{-\Omega}$. At the Hardy-space
level, related decomposition phenomena have been studied on the unit
circle and on $\mathbb R^n$; see
\cite{LiDengQian2016,DengLiQianCVEE2019}. Related decomposition aspects
also appear in the Hardy--Sobolev literature; see, for example,
\cite{DangQianYou2011}. In this section we show that the corresponding
boundary Sobolev functions admit a canonical Hardy--Sobolev
decomposition into two holomorphic components associated with these two
tube domains.

For the opposite tube domain $T_{-\Omega}$, we use the analogous
Hardy--Sobolev space defined through Fourier transforms supported in $-\Omega^\ast$,
equivalently through the reflected weight $w_n(-\xi)$.

\subsection{Boundary Sobolev space}

We first introduce the natural boundary Sobolev space associated with
the pair of opposite tube domains $T_\Omega$ and $T_{-\Omega}$. This is
the boundary-side class in which the forthcoming Hardy--Sobolev
decomposition will be formulated.

Let
\[
\widetilde{\Omega}^{\ast} := \Omega^\ast \cup (-\Omega^\ast).
\]
Since $\Omega$ is proper, the cones $\Omega^\ast$ and $-\Omega^\ast$ are
disjoint. Hence the reflected weight defined below is well posed.

We introduce the reflected weight
\[
\widetilde w_n(\xi) :=
\begin{cases}
w_n(\xi), & \xi\in \Omega^\ast,\\
w_n(-\xi), & \xi\in -\Omega^\ast .
\end{cases}
\]

\begin{definition}
Let $n\in\mathbb N$. We define
\[
W^n_{2,\rho}(\mathbb R^d;\widetilde{\Omega}^\ast)
\]
to be the space of all $u\in L^2(\mathbb R^d)$ such that
\[
\operatorname{ess\,supp}\widehat u \subset \widetilde{\Omega}^\ast
\]
and
\[
\|u\|^2_{W^n_{2,\rho}(\mathbb R^d;\widetilde{\Omega}^\ast)}
:=
\int_{\widetilde{\Omega}^\ast}
|\widehat u(\xi)|^2
\widetilde w_n(\xi)\,d\xi
<\infty.
\]
\end{definition}

This space may be viewed as a Sobolev-type space whose Fourier spectrum
is restricted to the union $\Omega^\ast\cup(-\Omega^\ast)$. Such a
spectral restriction is natural in the present setting, since the
Hardy--Sobolev spaces on the tube domains $T_\Omega$ and $T_{-\Omega}$
correspond to Fourier spectra supported in $\Omega^\ast$ and
$-\Omega^\ast$, respectively.

\subsection{Spectral decomposition}

The next step is to decompose the boundary data according to the two
opposite spectral cones $\Omega^\ast$ and $-\Omega^\ast$. This Fourier
decomposition is the boundary counterpart of the holomorphic
Hardy--Sobolev decomposition to be obtained below.

\begin{lemma}[Spectral decomposition]\label{lem:spec-decomp}
Let $u\in W^n_{2,\rho}(\mathbb R^d;\widetilde\Omega^\ast)$, and define
\[
\widehat u_+(\xi):=\chi_{\Omega^\ast}(\xi)\widehat u(\xi),
\qquad
\widehat u_-(\xi):=\chi_{-\Omega^\ast}(\xi)\widehat u(\xi).
\]
Then
\[
u=u_++u_-
\qquad \text{in }L^2(\mathbb R^d),
\]
and
\[
\operatorname{ess\,supp}\widehat u_+\subset\Omega^\ast,
\qquad
\operatorname{ess\,supp}\widehat u_-\subset-\Omega^\ast.
\]
Moreover,
\[
\widehat u_+\in L^2(\Omega^\ast,w_n),
\qquad
\widehat u_-(-\cdot)\in L^2(\Omega^\ast,w_n),
\]
and
\[
\|u\|^2_{W^n_{2,\rho}(\mathbb R^d;\widetilde{\Omega}^\ast)}
=
\int_{\Omega^\ast}|\widehat u_+(\xi)|^2 w_n(\xi)\,d\xi
+
\int_{\Omega^\ast}|\widehat u_-(-\xi)|^2 w_n(\xi)\,d\xi.
\]
\end{lemma}

\begin{proof}
Since
\[
\operatorname{ess\,supp}\widehat u \subset \Omega^\ast\cup(-\Omega^\ast)
\]
and the cones $\Omega^\ast$ and $-\Omega^\ast$ are disjoint, we have
\[
\widehat u = \widehat u_+ + \widehat u_-
\qquad \text{a.e. on }\mathbb R^d.
\]
Hence
\[
u=u_+ + u_-
\qquad \text{in }L^2(\mathbb R^d),
\]
by the injectivity of the Fourier transform on $L^2(\mathbb R^d)$. The
support properties are immediate from the definitions of $\widehat u_+$
and $\widehat u_-$.

Next,
\[
\int_{\Omega^\ast}|\widehat u_+(\xi)|^2 w_n(\xi)\,d\xi
\le
\int_{\widetilde\Omega^\ast}|\widehat u(\xi)|^2 \widetilde w_n(\xi)\,d\xi
<\infty,
\]
so $\widehat u_+\in L^2(\Omega^\ast,w_n)$. Likewise,
\[
\int_{\Omega^\ast}|\widehat u_-(-\xi)|^2 w_n(\xi)\,d\xi
=
\int_{-\Omega^\ast}|\widehat u_-(\eta)|^2 w_n(-\eta)\,d\eta
\le
\int_{\widetilde\Omega^\ast}|\widehat u(\eta)|^2 \widetilde w_n(\eta)\,d\eta
<\infty,
\]
so $\widehat u_-(-\cdot)\in L^2(\Omega^\ast,w_n)$.

Finally, since the two pieces are supported in the disjoint cones
$\Omega^\ast$ and $-\Omega^\ast$, the boundary norm decomposes as
\[
\|u\|^2_{W^n_{2,\rho}(\mathbb R^d;\widetilde{\Omega}^\ast)}
=
\int_{\Omega^\ast}|\widehat u_+(\xi)|^2 w_n(\xi)\,d\xi
+
\int_{-\Omega^\ast}|\widehat u_-(\eta)|^2 w_n(-\eta)\,d\eta.
\]
After the change of variables $\eta=-\xi$ in the second term, this
becomes exactly the asserted identity.
\end{proof}

\subsection{Hardy--Sobolev decomposition}

We now combine the spectral decomposition with the Paley--Wiener
characterization from Section~4 to lift each boundary spectral component
to a unique Hardy--Sobolev function on the corresponding tube domain.
This yields the desired Hardy--Sobolev decomposition and the associated
norm identity.

\begin{theorem}[Hardy--Sobolev decomposition]\label{thm:HS-decomp}
Let $n\in\mathbb N$ and let
\[
u\in W^n_{2,\rho}(\mathbb R^d;\widetilde{\Omega}^\ast).
\]
Let $u_+$ and $u_-$ be the spectral components of $u$ given by
Lemma~\ref{lem:spec-decomp}, that is,
\[
\widehat u_+(\xi)=\chi_{\Omega^\ast}(\xi)\widehat u(\xi),
\qquad
\widehat u_-(\xi)=\chi_{-\Omega^\ast}(\xi)\widehat u(\xi).
\]
Then there exist unique functions
\[
F_+\in H^n_{2,\rho}(T_\Omega),
\qquad
F_-\in H^n_{2,\rho}(T_{-\Omega}),
\]
such that

\begin{enumerate}
\item[(i)] $F_+$ and $F_-$ admit $L^2$ boundary values $u_+$ and $u_-$,
respectively;

\item[(ii)] the boundary decomposition
\[
u=u_+ + u_-
\]
holds in $L^2(\mathbb R^d)$;

\item[(iii)] the norm identity
\[
\|u\|^2_{W^n_{2,\rho}(\mathbb R^d;\widetilde{\Omega}^\ast)}
=
\|F_+\|^2_{H^n_{2,\rho}(T_\Omega)}
+
\|F_-\|^2_{H^n_{2,\rho}(T_{-\Omega})}
\]
is satisfied.
\end{enumerate}
\end{theorem}

\begin{proof}
Let $u_+$ and $u_-$ be as in Lemma~\ref{lem:spec-decomp}. Since
\[
\widehat u_+\in L^2(\Omega^\ast,w_n),
\]
the Paley--Wiener characterization established in Section~4 yields a
unique function $F_+\in H^n_{2,\rho}(T_\Omega)$ given by
\[
F_+(z)
=
\int_{\Omega^\ast}
e^{i\langle z,\xi\rangle}\widehat u_+(\xi)\,d\xi,
\qquad z\in T_\Omega.
\]

Similarly, because
\[
\widehat u_-(-\cdot)\in L^2(\Omega^\ast,w_n),
\]
the same characterization produces a unique function
\[
G(z)
=
\int_{\Omega^\ast}
e^{i\langle z,\xi\rangle}\widehat u_-(-\xi)\,d\xi,
\qquad z\in T_\Omega,
\]
belonging to $H^n_{2,\rho}(T_\Omega)$. Define
\[
F_-(z):=G(-z), \qquad z\in T_{-\Omega}.
\]
Since $-z\in T_\Omega$ whenever $z\in T_{-\Omega}$, it follows that
$F_-$ is holomorphic on $T_{-\Omega}$. Equivalently, after the change of
variables $\eta=-\xi$, we obtain the representation
\[
F_-(z)
=
\int_{-\Omega^\ast}
e^{i\langle z,\eta\rangle}\widehat u_-(\eta)\,d\eta,
\qquad z\in T_{-\Omega}.
\]
Thus $F_-\in H^n_{2,\rho}(T_{-\Omega})$.

We now verify the boundary behaviour. Let $y\in\Omega$. By the Fourier
description of boundary translates,
\[
\widehat{F_+(\cdot+iy)}(\xi)
=
e^{-\langle y,\xi\rangle}\widehat u_+(\xi)
\qquad \text{for a.e. }\xi\in\mathbb R^d.
\]
Therefore
\[
\|F_+(\cdot+iy)-u_+\|_{L^2(\mathbb R^d)}^2
=
\int_{\Omega^\ast}
|e^{-\langle y,\xi\rangle}-1|^2
|\widehat u_+(\xi)|^2\,d\xi .
\]
Since
\[
|e^{-\langle y,\xi\rangle}-1|^2 \le 4
\qquad (\xi\in\Omega^\ast),
\]
the integrand is dominated by
\[
4|\widehat u_+(\xi)|^2\in L^1(\Omega^\ast).
\]
Moreover, for each fixed $\xi\in\Omega^\ast$,
\[
e^{-\langle y,\xi\rangle}\to 1
\qquad \text{as } y\to 0 \text{ within }\Omega.
\]
Hence, by the dominated convergence theorem,
\[
F_+(\cdot+iy)\to u_+
\qquad \text{in }L^2(\mathbb R^d).
\]

An analogous argument for $F_-$ shows that
\[
\widehat{F_-(\cdot-iy)}(\eta)
=
e^{\langle y,\eta\rangle}\widehat u_-(\eta)
\qquad \text{for a.e. }\eta\in -\Omega^\ast,
\]
and therefore
\[
\|F_-(\cdot-iy)-u_-\|_{L^2(\mathbb R^d)}^2
=
\int_{-\Omega^\ast}
|e^{\langle y,\eta\rangle}-1|^2
|\widehat u_-(\eta)|^2\,d\eta \to 0
\]
as $y\to 0$ within $\Omega$, again by dominated convergence. Thus
$F_-$ admits $u_-$ as its $L^2$ boundary value on $T_{-\Omega}$.

The decomposition
\[
u=u_+ + u_-
\]
is exactly the conclusion of Lemma~\ref{lem:spec-decomp}.

Finally, by the isometric property of the Fourier--Laplace transform,
\[
\|F_+\|^2_{H^n_{2,\rho}(T_\Omega)}
=
\int_{\Omega^\ast}
|\widehat u_+(\xi)|^2 w_n(\xi)\,d\xi,
\]
and
\[
\|F_-\|^2_{H^n_{2,\rho}(T_{-\Omega})}
=
\int_{\Omega^\ast}
|\widehat u_-(-\xi)|^2 w_n(\xi)\,d\xi .
\]
Combining these identities with the norm formula in
Lemma~\ref{lem:spec-decomp} yields
\[
\|u\|^2_{W^n_{2,\rho}(\mathbb R^d;\widetilde{\Omega}^\ast)}
=
\|F_+\|^2_{H^n_{2,\rho}(T_\Omega)}
+
\|F_-\|^2_{H^n_{2,\rho}(T_{-\Omega})}.
\]

The uniqueness of $F_+$ and $F_-$ follows from the uniqueness part of
the Paley--Wiener characterization, since their boundary spectral
densities are uniquely determined by $\widehat u_+$ and $\widehat u_-$,
respectively. This completes the proof.
\end{proof}

The above decomposition should be contrasted with the orthant setting,
where one may consider all $2^d$ sign orthants in frequency space and
obtain a multi-component Hardy decomposition. In the present convex-cone
framework the geometry of the tube domain naturally leads to the
two-component splitting associated with $\Omega^\ast$ and $-\Omega^\ast$.

\section{Reproducing kernels}

Having established that $H^n_{2,\rho}(T_\Omega)$ is a Hilbert space of
holomorphic functions with bounded point evaluations, we now turn to its
reproducing-kernel structure. The kernel provides a concrete realization
of the Hilbert geometry of the space and will serve as a basic tool in
the subsequent analysis. In this section we derive an explicit kernel
formula from the weighted Fourier--Laplace model and record several of
its fundamental properties.

\subsection{Existence of the reproducing kernel}

We begin by recalling the abstract reproducing-kernel construction that
follows from bounded point evaluations. This provides the notation and
framework for the explicit formula derived next.

By Proposition~\ref{prop:HS-point-evaluation} and the Hilbert structure
established in Theorem~\ref{thm:HS-Fourier-Hilbert}, the point evaluation
functional
\[
\delta_w:F\mapsto F(w)
\]
is bounded on $H^n_{2,\rho}(T_\Omega)$ for every $w\in T_\Omega$.
Hence, by the Riesz representation theorem, there exists a unique
function
\[
K^{(n,\rho)}_w\in H^n_{2,\rho}(T_\Omega)
\]
such that
\[
F(w)=\langle F,K^{(n,\rho)}_w\rangle_{H^n_{2,\rho}(T_\Omega)}
\qquad
\text{for all }F\in H^n_{2,\rho}(T_\Omega).
\]
The reproducing kernel of $H^n_{2,\rho}(T_\Omega)$ is therefore given by
\[
K^{(n,\rho)}(z,w):=K^{(n,\rho)}_w(z),
\qquad z,w\in T_\Omega.
\]

\subsection{Reproducing kernel formula}

We now identify the reproducing kernel explicitly by transporting the
Riesz representer through the weighted Fourier--Laplace realization.

\begin{theorem}[Reproducing kernel formula]\label{thm:kernel-formula}
For $z,w\in T_\Omega$, the reproducing kernel of
$H^n_{2,\rho}(T_\Omega)$ is given by
\[
K^{(n,\rho)}(z,w)
=
\int_{\Omega^\ast}
\frac{e^{\,i\langle z-\overline w,\xi\rangle}}{w_n(\xi)}\,d\xi.
\]
Equivalently,
\[
K^{(n,\rho)}_w(z)
=
\int_{\Omega^\ast}
e^{\,i\langle z,\xi\rangle}
\frac{e^{-i\langle \overline w,\xi\rangle}}{w_n(\xi)}\,d\xi.
\]
\end{theorem}

\begin{proof}
Fix $w\in T_\Omega$ and define
\[
k_w(\xi):=\frac{e^{-i\langle \overline w,\xi\rangle}}{w_n(\xi)},
\qquad \xi\in\Omega^\ast.
\]
We first show that $k_w\in L^2(\Omega^\ast,w_n)$. Since $\Im w\in\Omega$,
there exists a constant $c_w>0$ such that
\[
\langle \Im w,\xi\rangle \ge c_w |\xi|,
\qquad \xi\in\Omega^\ast.
\]
Hence
\[
\int_{\Omega^\ast}|k_w(\xi)|^2w_n(\xi)\,d\xi
=
\int_{\Omega^\ast}
\frac{e^{-2\langle \Im w,\xi\rangle}}{w_n(\xi)}\,d\xi
\le
\int_{\Omega^\ast} e^{-2c_w|\xi|}\,d\xi
<\infty.
\]
Therefore $k_w\in L^2(\Omega^\ast,w_n)$.

Let $F=\mathcal{L}f\in H^n_{2,\rho}(T_\Omega)$ with
$f\in L^2(\Omega^\ast,w_n)$. Then
\[
F(w)
=
\int_{\Omega^\ast} e^{\,i\langle w,\xi\rangle}f(\xi)\,d\xi.
\]
On the other hand, by the transported inner product,
\[
\langle F,\mathcal{L}k_w\rangle_{H^n_{2,\rho}(T_\Omega)}
=
\int_{\Omega^\ast} f(\xi)\overline{k_w(\xi)}\,w_n(\xi)\,d\xi.
\]
Since
\[
\overline{k_w(\xi)}\,w_n(\xi)=e^{\,i\langle w,\xi\rangle},
\]
it follows that
\[
\langle F,\mathcal{L}k_w\rangle_{H^n_{2,\rho}(T_\Omega)}=F(w).
\]
By uniqueness in the Riesz representation theorem, we conclude that
\[
K^{(n,\rho)}_w=\mathcal{L}k_w.
\]
Therefore,
\[
K^{(n,\rho)}_w(z)
=
\int_{\Omega^\ast}
e^{\,i\langle z,\xi\rangle}
\frac{e^{-i\langle \overline w,\xi\rangle}}{w_n(\xi)}\,d\xi,
\]
and hence
\[
K^{(n,\rho)}(z,w)
=
\int_{\Omega^\ast}
\frac{e^{\,i\langle z-\overline w,\xi\rangle}}{w_n(\xi)}\,d\xi.
\]
This completes the proof.
\end{proof}

The explicit formula shows that the kernel depends only on the
combination $z-\overline w$, reflecting the translation-invariant
geometry of the tube domain together with the weighted Fourier--Laplace
realization of the space.

\subsection{Basic properties}

The explicit kernel formula immediately yields several structural
properties that reflect both the Hilbert-space symmetry and the
translation-invariant geometry of the tube domain.

\begin{proposition}\label{prop:kernel-basic}
For all $z,w\in T_\Omega$, the reproducing kernel
$K^{(n,\rho)}(z,w)$ satisfies the following properties:

\begin{enumerate}
\item \emph{Sesqui-holomorphy}: the function
\[
(z,w)\mapsto K^{(n,\rho)}(z,w)
\]
is holomorphic in $z$ and anti-holomorphic in $w$ on
$T_\Omega\times T_\Omega$.

\item \emph{Hermitian symmetry}:
\[
K^{(n,\rho)}(z,w)=\overline{K^{(n,\rho)}(w,z)}.
\]

\item \emph{Translation invariance in the real direction}: for every
$x\in\mathbb R^d$,
\[
K^{(n,\rho)}(z+x,w+x)=K^{(n,\rho)}(z,w).
\]

\item \emph{Diagonal formula}: if $z=x+iy\in T_\Omega$, then
\[
K^{(n,\rho)}(z,z)
=
\int_{\Omega^\ast}
\frac{e^{-2\langle y,\xi\rangle}}{w_n(\xi)}\,d\xi.
\]

\item \emph{Positivity on the diagonal}:
\[
K^{(n,\rho)}(z,z)>0.
\]
\end{enumerate}
\end{proposition}

\begin{proof}
(1) Let $K\subset T_\Omega\times T_\Omega$ be compact. Since
$\Im z,\Im w\in\Omega$ for all $(z,w)\in K$, there exists a constant
$c_K>0$ such that
\[
\langle \Im z+\Im w,\xi\rangle \ge c_K|\xi|,
\qquad (z,w)\in K,\ \xi\in\Omega^\ast.
\]
Hence
\[
\left|
\frac{e^{\,i\langle z-\overline w,\xi\rangle}}{w_n(\xi)}
\right|
=
\frac{e^{-\langle \Im z+\Im w,\xi\rangle}}{w_n(\xi)}
\le
e^{-c_K|\xi|},
\]
which is integrable on $\Omega^\ast$. Therefore the kernel formula in
Theorem~\ref{thm:kernel-formula} defines a function that is holomorphic
in $z$ and anti-holomorphic in $w$ on $T_\Omega\times T_\Omega$.

(2) By the kernel formula,
\[
\overline{K^{(n,\rho)}(w,z)}
=
\overline{
\int_{\Omega^\ast}
\frac{e^{\,i\langle w-\overline z,\xi\rangle}}{w_n(\xi)}\,d\xi
}
=
\int_{\Omega^\ast}
\frac{e^{\,i\langle z-\overline w,\xi\rangle}}{w_n(\xi)}\,d\xi
=
K^{(n,\rho)}(z,w).
\]

(3) For $x\in\mathbb R^d$,
\[
K^{(n,\rho)}(z+x,w+x)
=
\int_{\Omega^\ast}
\frac{e^{\,i\langle z+x-\overline{(w+x)},\xi\rangle}}{w_n(\xi)}\,d\xi
=
\int_{\Omega^\ast}
\frac{e^{\,i\langle z-\overline w,\xi\rangle}}{w_n(\xi)}\,d\xi
=
K^{(n,\rho)}(z,w).
\]

(4) This is the special case $w=z$ of the kernel formula in
Theorem~\ref{thm:kernel-formula}.

(5) By (iv),
\[
K^{(n,\rho)}(z,z)
=
\int_{\Omega^\ast}
\frac{e^{-2\langle \Im z,\xi\rangle}}{w_n(\xi)}\,d\xi.
\]
The integrand is strictly positive on $\Omega^\ast$, and $\Omega^\ast$
has positive measure. Hence
\[
K^{(n,\rho)}(z,z)>0.
\]
\end{proof}

\subsection{Derivative kernels}

The integral representation of the reproducing kernel is stable under
differentiation. Consequently, one obtains explicit formulas for mixed
derivatives of the kernel, which reflect the same weighted
Fourier--Laplace structure that defines the space itself.

\begin{proposition}\label{prop:kernel-derivatives}
Let $\alpha,\beta$ be multi-indices. Then for every $z,w\in T_\Omega$,
\[
\partial_z^\alpha \partial_{\overline w}^{\,\beta}
K^{(n,\rho)}(z,w)
=
\int_{\Omega^\ast}
(i\xi)^\alpha(-i\xi)^\beta
\frac{e^{\,i\langle z-\overline w,\xi\rangle}}{w_n(\xi)}\,d\xi.
\]
Moreover, these mixed derivatives are jointly continuous on
$T_\Omega\times T_\Omega$.
\end{proposition}

\begin{proof}
Since differentiability and continuity are local questions on
$T_\Omega\times T_\Omega$, it suffices to work on an arbitrary compact
subset $K\subset T_\Omega\times T_\Omega$. Since $\Im z,\Im w\in\Omega$
for all $(z,w)\in K$, there exists a constant $c_K>0$ such that
\[
\langle \Im z+\Im w,\xi\rangle \ge c_K|\xi|,
\qquad (z,w)\in K,\ \xi\in\Omega^\ast.
\]
Hence for every pair of multi-indices $\alpha,\beta$,
\[
\left|
(i\xi)^\alpha(-i\xi)^\beta
\frac{e^{\,i\langle z-\overline w,\xi\rangle}}{w_n(\xi)}
\right|
\le
C_{K,\alpha,\beta}(1+|\xi|)^{|\alpha|+|\beta|}e^{-c_K|\xi|},
\]
where we used $w_n(\xi)\ge 1$. The right-hand side is integrable on
$\Omega^\ast$. Therefore differentiation under the integral sign is
justified, and
\[
\partial_z^\alpha \partial_{\overline w}^{\,\beta}
K^{(n,\rho)}(z,w)
=
\int_{\Omega^\ast}
(i\xi)^\alpha(-i\xi)^\beta
\frac{e^{\,i\langle z-\overline w,\xi\rangle}}{w_n(\xi)}\,d\xi
\]
for all $(z,w)\in T_\Omega\times T_\Omega$. The same domination
argument on compact subsets yields the joint continuity of these mixed
derivatives.
\end{proof}

The explicit reproducing  kernel and derivative-kernel formulas show that the
reproducing structure of $H^n_{2,\rho}(T_\Omega)$ is governed by the
same weighted Fourier--Laplace representation of the space. In
particular, the cone geometry enters through the exponential factor
$e^{-\langle \Im z+\Im w,\xi\rangle}$, while the Hardy--Sobolev
regularity appears through  the weight $w_n(\xi)^{-1}$. These formulas
provide a natural starting point for further kernel estimates and
operator-theoretic applications.

\section{Carleson measures}

The reproducing-kernel representation established in Section~6 provides
the natural framework for formulating the Carleson embedding problem for
$H^n_{2,\rho}(T_\Omega)$. In this section we introduce Carleson
measures for the Hardy--Sobolev space, recast the embedding condition in
operator-theoretic and Fourier--Laplace terms, and derive the natural
reproducing-kernel testing criterion. These results prepare the ground
for the operator-theoretic applications developed in Section~8.

We begin with the natural Carleson embedding condition for
$H^n_{2,\rho}(T_\Omega)$ and record its immediate interpretation as a
bounded inclusion into $L^2(T_\Omega,\mu)$.

\begin{definition}
Let $\mu$ be a positive Borel measure on $T_\Omega$.
We say that $\mu$ is a \emph{Carleson measure} for
$H^n_{2,\rho}(T_\Omega)$ if there exists a constant $C>0$ such that
\[
\int_{T_\Omega}|F(z)|^2\,d\mu(z)
\le
C\|F\|^2_{H^n_{2,\rho}(T_\Omega)}
\]
for all $F\in H^n_{2,\rho}(T_\Omega)$.
\end{definition}

Equivalently, $\mu$ is a Carleson measure if and only if the canonical
embedding
\[
J_\mu:H^n_{2,\rho}(T_\Omega)\longrightarrow L^2(T_\Omega,\mu),
\qquad
J_\mu(F)=F,
\]
is bounded.

\subsection{Operator-theoretic formulation}

The Carleson inequality may be expressed equivalently in terms of a
bounded sesquilinear form and the associated positive operator on
$H^n_{2,\rho}(T_\Omega)$. This is the operator-theoretic viewpoint that
will be used later.

Associated with $\mu$ is the sesquilinear form
\[
B_\mu(F,G)
:=
\int_{T_\Omega}F(z)\overline{G(z)}\,d\mu(z),
\qquad
F,G\in H^n_{2,\rho}(T_\Omega).
\]

\begin{proposition}\label{prop:carleson-operator}
Let $\mu$ be a positive Borel measure on $T_\Omega$.
The following conditions are equivalent.

\begin{enumerate}
\item[(i)] $\mu$ is a Carleson measure for $H^n_{2,\rho}(T_\Omega)$;

\item[(ii)] the embedding
\[
J_\mu:H^n_{2,\rho}(T_\Omega)\to L^2(T_\Omega,\mu)
\]
is bounded;

\item[(iii)] the sesquilinear form $B_\mu$ is bounded on
$H^n_{2,\rho}(T_\Omega)\times H^n_{2,\rho}(T_\Omega)$;

\item[(iv)] there exists a unique bounded positive operator
\[
T_\mu:H^n_{2,\rho}(T_\Omega)\to H^n_{2,\rho}(T_\Omega)
\]
such that
\[
\langle T_\mu F,G\rangle_{H^n_{2,\rho}(T_\Omega)}
=
B_\mu(F,G)
\]
for all $F,G\in H^n_{2,\rho}(T_\Omega)$.
\end{enumerate}
\end{proposition}

\begin{proof}
The equivalence of \textup{(i)} and \textup{(ii)} is immediate from the
definition of the Carleson inequality. The equivalence of \textup{(ii)}
and \textup{(iii)} follows from the identity
\[
B_\mu(F,G)
=
\langle J_\mu F,J_\mu G\rangle_{L^2(T_\Omega,\mu)}.
\]
Finally, the equivalence of \textup{(iii)} and \textup{(iv)} follows
from the Riesz representation theorem for bounded sesquilinear forms on
Hilbert spaces. Moreover, since $\mu$ is positive, one has
\[
B_\mu(F,F)
=
\int_{T_\Omega}|F(z)|^2\,d\mu(z)\ge 0,
\]
so the representing operator $T_\mu$ is positive.
\end{proof}

\subsection{Fourier--Laplace formulation}

Since $H^n_{2,\rho}(T_\Omega)$ is realized isometrically through the
Fourier--Laplace transform, the Carleson embedding problem can also be
transferred to the weighted Fourier side.

Let
\[
\mathcal L:L^2(\Omega^\ast,w_n)\longrightarrow H^n_{2,\rho}(T_\Omega)
\]
denote the Fourier--Laplace transform introduced in Section~3.

\begin{theorem}\label{thm:carleson-fourier}
Let $\mu$ be a positive Borel measure on $T_\Omega$.
Then $\mu$ is a Carleson measure for $H^n_{2,\rho}(T_\Omega)$
if and only if there exists a constant $C>0$ such that
\[
\int_{T_\Omega}|(\mathcal L f)(z)|^2\,d\mu(z)
\le
C\|f\|^2_{L^2(\Omega^\ast,w_n)}
\]
for all $f\in L^2(\Omega^\ast,w_n)$.
\end{theorem}

\begin{proof}
Since
\[
\mathcal L:L^2(\Omega^\ast,w_n)\to H^n_{2,\rho}(T_\Omega)
\]
is an isometric isomorphism, the Carleson inequality on
$H^n_{2,\rho}(T_\Omega)$ holds for all $F\in H^n_{2,\rho}(T_\Omega)$ if
and only if it holds for all $F=\mathcal L f$ with
$f\in L^2(\Omega^\ast,w_n)$. This is exactly the stated condition.
\end{proof}

\subsection{Kernel testing condition}

The reproducing kernels obtained in Section~6 yield an immediate
necessary condition for Carleson measures. This is the natural kernel
test associated with the embedding inequality.

\begin{proposition}[Kernel test]\label{prop:kernel-test}
Let $\mu$ be a Carleson measure for $H^n_{2,\rho}(T_\Omega)$.
Then there exists a constant $C>0$ such that
\[
\int_{T_\Omega}|K^{(n,\rho)}_w(z)|^2\,d\mu(z)
\le
C\,K^{(n,\rho)}(w,w)
\]
for all $w\in T_\Omega$.
\end{proposition}

\begin{proof}
Applying the Carleson inequality to
\[
F=K^{(n,\rho)}_w
\]
gives
\[
\int_{T_\Omega}|K^{(n,\rho)}_w(z)|^2\,d\mu(z)
\le
C\|K^{(n,\rho)}_w\|^2_{H^n_{2,\rho}(T_\Omega)}.
\]
Since
\[
\|K^{(n,\rho)}_w\|^2_{H^n_{2,\rho}(T_\Omega)}
=
K^{(n,\rho)}(w,w),
\]
the asserted inequality follows.
\end{proof}

The diagonal kernel formula obtained in Section~6 shows that for
$w=x+iy\in T_\Omega$,
\[
K^{(n,\rho)}(w,w)
=
\int_{\Omega^\ast}
\frac{e^{-2\langle y,\xi\rangle}}{w_n(\xi)}\,d\xi.
\]
Thus the kernel test is governed by the decay of the exponential factor
along the dual cone together with the Hardy--Sobolev weight
$w_n(\xi)$.

In the setting of general tube domains over convex cones, a simple
geometric characterization of Carleson measures comparable to the
classical Carleson-box condition is not available in the present
framework. For this reason we adopt the operator-theoretic formulation
above and work primarily with the embedding condition, its
Fourier--Laplace reformulation, and the kernel testing inequality.

\section{Preliminary operator-theoretic consequences}

We conclude by recording several basic operator-theoretic consequences
of the reproducing-kernel and Carleson-measure framework developed in
the preceding sections. Our purpose here is not to give a full theory of
multipliers and weighted composition operators on
$H^n_{2,\rho}(T_\Omega)$, but rather to isolate those structural facts
that follow directly from the reproducing-kernel formalism and that
suggest natural directions for further study.

\subsection{Multipliers}

We begin with multiplication operators, which form the simplest class of
operators naturally associated with the function-space structure of
$H^n_{2,\rho}(T_\Omega)$. Even at this preliminary level, the
reproducing-kernel point of view already yields useful necessary
conditions for boundedness.

\begin{definition}
The \emph{multiplier algebra} of $H^n_{2,\rho}(T_\Omega)$ is defined by
\[
\mathcal M(H^n_{2,\rho}(T_\Omega))
=
\left\{
\psi\in H(T_\Omega):
\psi F\in H^n_{2,\rho}(T_\Omega)
\text{ for all }F\in H^n_{2,\rho}(T_\Omega)
\right\}.
\]
For $\psi\in \mathcal M(H^n_{2,\rho}(T_\Omega))$, the associated
multiplication operator
\[
M_\psi:H^n_{2,\rho}(T_\Omega)\to H^n_{2,\rho}(T_\Omega)
\]
is given by
\[
M_\psi F=\psi F.
\]
\end{definition}

\begin{proposition}\label{prop:multiplier-basic}
Let $\psi\in \mathcal M(H^n_{2,\rho}(T_\Omega))$. Then $M_\psi$ is a
bounded operator on $H^n_{2,\rho}(T_\Omega)$. Moreover, for every
$w\in T_\Omega$,
\[
M_\psi^*K_w^{(n,\rho)}
=
\overline{\psi(w)}\,K_w^{(n,\rho)}.
\]
Consequently,
\[
|\psi(w)|\le \|M_\psi\|,
\qquad w\in T_\Omega,
\]
and in particular
\[
\psi\in H^\infty(T_\Omega)
\qquad\text{with}\qquad
\|\psi\|_{H^\infty(T_\Omega)}\le \|M_\psi\|.
\]
\end{proposition}

\begin{proof}
By definition, $M_\psi$ is a linear operator from
$H^n_{2,\rho}(T_\Omega)$ into itself with domain equal to the whole
space. We claim that its graph is closed. Indeed, let
\[
F_j\to F,\qquad M_\psi F_j\to G
\]
in $H^n_{2,\rho}(T_\Omega)$. Since point evaluations are bounded on
$H^n_{2,\rho}(T_\Omega)$, both convergences imply local uniform
convergence on $T_\Omega$. Hence
\[
G(z)=\lim_{j\to\infty}(M_\psi F_j)(z)
=
\lim_{j\to\infty}\psi(z)F_j(z)
=
\psi(z)F(z)
\]
for every $z\in T_\Omega$. Thus $G=M_\psi F$, and the graph of $M_\psi$
is closed. By the closed graph theorem, $M_\psi$ is bounded.

Now let $w\in T_\Omega$ and $F\in H^n_{2,\rho}(T_\Omega)$. Using the
reproducing property,
\[
\langle F,M_\psi^*K_w^{(n,\rho)}\rangle
=
\langle M_\psi F,K_w^{(n,\rho)}\rangle
=
(M_\psi F)(w)
=
\psi(w)F(w)
=
\psi(w)\langle F,K_w^{(n,\rho)}\rangle.
\]
Therefore
\[
\langle F,M_\psi^*K_w^{(n,\rho)}\rangle
=
\langle F,\overline{\psi(w)}\,K_w^{(n,\rho)}\rangle
\]
for all $F$, and hence
\[
M_\psi^*K_w^{(n,\rho)}
=
\overline{\psi(w)}\,K_w^{(n,\rho)}.
\]

Taking norms, we obtain
\[
|\psi(w)|\,\|K_w^{(n,\rho)}\|
=
\|M_\psi^*K_w^{(n,\rho)}\|
\le
\|M_\psi\|\,\|K_w^{(n,\rho)}\|.
\]
Since $\|K_w^{(n,\rho)}\|>0$, it follows that
\[
|\psi(w)|\le \|M_\psi\|.
\]
Taking the supremum over $w\in T_\Omega$ gives
\[
\|\psi\|_{H^\infty(T_\Omega)}\le \|M_\psi\|.
\]
\end{proof}

\subsection{Weighted composition operators}

We next consider weighted composition operators, which generalize both
multiplication operators and classical composition operators. The
reproducing-kernel formalism gives an immediate description of their
adjoint action on kernel functions and leads to natural necessary
boundedness conditions.

Let $\varphi:T_\Omega\to T_\Omega$ be a holomorphic self-map and let
$\psi\in H(T_\Omega)$.

\begin{definition}
The \emph{weighted composition operator} associated with $(\psi,\varphi)$
is the formal map
\[
W_{\psi,\varphi}F
=
\psi\cdot(F\circ\varphi),
\qquad F\in H^n_{2,\rho}(T_\Omega),
\]
whenever this expression defines an element of
$H^n_{2,\rho}(T_\Omega)$.
\end{definition}

\begin{proposition}\label{prop:wco-adjoint}
Let $W_{\psi,\varphi}$ be a bounded weighted composition operator on
$H^n_{2,\rho}(T_\Omega)$. Then for every $w\in T_\Omega$,
\[
W_{\psi,\varphi}^*K_w^{(n,\rho)}
=
\overline{\psi(w)}\,K_{\varphi(w)}^{(n,\rho)}.
\]
\end{proposition}

\begin{proof}
Let $F\in H^n_{2,\rho}(T_\Omega)$. Then
\[
\langle F,W_{\psi,\varphi}^*K_w^{(n,\rho)}\rangle
=
\langle W_{\psi,\varphi}F,K_w^{(n,\rho)}\rangle
=
(W_{\psi,\varphi}F)(w).
\]
By definition of $W_{\psi,\varphi}$,
\[
(W_{\psi,\varphi}F)(w)
=
\psi(w)\,F(\varphi(w)).
\]
Using the reproducing property again,
\[
F(\varphi(w))
=
\langle F,K_{\varphi(w)}^{(n,\rho)}\rangle.
\]
Hence
\[
\langle F,W_{\psi,\varphi}^*K_w^{(n,\rho)}\rangle
=
\psi(w)\,\langle F,K_{\varphi(w)}^{(n,\rho)}\rangle
=
\langle F,\overline{\psi(w)}\,K_{\varphi(w)}^{(n,\rho)}\rangle.
\]
Since this identity holds for all $F\in H^n_{2,\rho}(T_\Omega)$, the
result follows.
\end{proof}

\subsection{Kernel necessary conditions}

The adjoint action on reproducing kernels immediately yields natural
necessary conditions for boundedness. These conditions are the most
robust operator-theoretic consequences available at the present stage.

\begin{corollary}[Kernel necessary condition]\label{cor:wco-kernel-necessary}
Let $W_{\psi,\varphi}$ be a bounded weighted composition operator on
$H^n_{2,\rho}(T_\Omega)$. Then
\[
|\psi(w)|^2\,K^{(n,\rho)}(\varphi(w),\varphi(w))
\le
\|W_{\psi,\varphi}\|^2\,K^{(n,\rho)}(w,w)
\]
for every $w\in T_\Omega$.
\end{corollary}

\begin{proof}
By Proposition~\ref{prop:wco-adjoint},
\[
W_{\psi,\varphi}^*K_w^{(n,\rho)}
=
\overline{\psi(w)}\,K_{\varphi(w)}^{(n,\rho)}.
\]
Taking norms and using the reproducing-kernel identity
\[
\|K_w^{(n,\rho)}\|^2=K^{(n,\rho)}(w,w),
\]
we obtain
\[
\begin{aligned}
|\psi(w)|^2\,K^{(n,\rho)}(\varphi(w),\varphi(w))
&=
\|W_{\psi,\varphi}^*K_w^{(n,\rho)}\|^2 \\
&\le
\|W_{\psi,\varphi}\|^2\,\|K_w^{(n,\rho)}\|^2 \\
&=
\|W_{\psi,\varphi}\|^2\,K^{(n,\rho)}(w,w).
\end{aligned}
\]
This proves the claim.
\end{proof}

The preceding results show that the reproducing-kernel and
Carleson-measure methods developed in the earlier sections provide a
natural framework for the operator theory of
$H^n_{2,\rho}(T_\Omega)$. In particular, kernel estimates yield
necessary conditions for the boundedness of weighted composition
operators and pointwise control for multipliers. A systematic treatment
of multiplier algebras and weighted composition operators, including
sharp boundedness and compactness criteria, is not pursued in the
present paper and will be studied elsewhere.

\section*{Acknowledgements}
	This work was supported by the NSF of Guangdong Province, China (Grant No. 2025A1515011213) and the Science and Technology Development Fund of Macau SAR (No. 0020/2023/RIB1) . 

\bibliographystyle{plain}
\bibliography{HS}

\end{document}